\documentclass[reqno]{amsart}
\usepackage{amsmath}
\usepackage{amssymb}
\usepackage{latexsym}
\usepackage{amsthm}
\usepackage{hyperref}

\begin{document}
\title
{Existence results for degenerate $p(x)$-Laplace equations with Leray-Lions type operators}

\author
{Ky Ho and Inbo Sim}
\address{Ky Ho \newline
Department of Mathematics, University of Ulsan, Ulsan 680-749, Republic of Korea}
\email{hnky81@gmail.com}
\address{Inbo Sim \newline
Department of Mathematics, University of Ulsan, Ulsan 680-749, Republic of Korea}
\email{ibsim@ulsan.ac.kr}


\begin{abstract}
 We show the various existence results for degenerate $p(x)$-Laplace equations with Leray-Lions type operators. A suitable condition on degeneracy is discussed and proofs are mainly based on direct methods and critical point theories in Calculus of Variations. In particular, we investigate the various situations of the growth rates between principal operators and nonlinearities.

\end{abstract}

\maketitle
\numberwithin{equation}{section}
\newtheorem{theorem}{Theorem}[section]
\newtheorem{lemma}[theorem]{Lemma}
\newtheorem{proposition}[theorem]{Proposition}
\newtheorem{corollary}[theorem]{Corollary}
\newtheorem{definition}[theorem]{Definition}
\newtheorem{example}[theorem]{Example}
\newtheorem{remark}[theorem]{Remark}
\allowdisplaybreaks

\section{Introduction}
Many problems in various real fields, for examples, electro-rheological
fluids \cite{Ruzicka}, the thermistor
problem \cite{Zhikov}, or the problem of image recovery \cite{Chen} are related to differential equations with nonsmooth growth which contain $p(x)$-Laplacian. In this paper, we investigate  the following equations which extend
 $p(x)$-Laplacian 
\begin{eqnarray}\label{1.1}
\begin{cases}
 -\operatorname{div}a(x,\nabla u)=\lambda f(x,u) \quad &\text{in } \Omega ,\\
  u=0\quad &\text{on } \partial \Omega ,
  \end{cases}
\end{eqnarray}
where $\Omega $ is a bounded domain in $\mathbb{R}^{N}$ with a Lipschitz boundary $\partial \Omega;$ $a : \Omega \times \mathbb R^N \to
\mathbb R^N$ and $f : \Omega \times \mathbb R \to
\mathbb R$ are Carath\'{e}odory functions; and $\lambda$ is a positive parameter. The operator $ \operatorname{div}a(x,\nabla u)$ generalizes a degenerate $p(x)$-Laplacian $ \operatorname{div}\left( w(x)| \nabla u| ^{p(x)-2}\nabla u\right),$ where $p\in C(\overline{\Omega},(1,\infty)) =: C_+(\overline{\Omega})$ and $w\in P_+(\Omega)$ the set of all measurable functions on $\Omega$ that are positive a.e. on $\Omega$. There have been many studies about $p(x)$-Laplacian (see \cite {Boureanu, Diening, Fan3, Fan4, Kovacik, Mihailescu1} and the references therein).

Throughout the paper, we assume that $w\in P_+(\Omega)$ and $p\in C_+(\overline{\Omega});$ and for each $q\in C_+(\overline{\Omega})$, set $q^- := \min_{x \in \overline{\Omega}} q(x)$, $q^+ := \max_{x \in \overline{\Omega}} q(x),$ and let $q'$ denote the conjugate function of $q$, i.e., $\frac{1}{q(x)}+\frac{1}{q'(x)}=1, \forall x\in \overline{\Omega}$. Furthermore, we assume that
\begin{itemize}
\item[($A0$)]
$a(x,-\xi)=-a(x,\xi)$ for a.e. $x\in \Omega $ and all $\xi \in \mathbb R^N.$

\item[($A1)$] There exists a Carath\'{e}odory function $A : \Omega \times \mathbb R^N \to \mathbb R,$ continuously differentiable with respect to its second argument, such that $A(x,0)=0$ for a.e. $x\in \Omega$ and $a(x,\xi)=\nabla_\xi A(x,\xi)$ for a.e. $x\in \Omega $ and all $\xi \in \mathbb R^N.$

\item[($A2$)] $|a(x,\xi)| \leq \widetilde{C} w(x)\left[k(x) + |\xi|^{p(x)-1}\right]$
for a.e. $x\in \Omega $ and all $\xi \in \mathbb R^N ,$ where $\widetilde{C}$ is a positive constant, $k\in P_+(\Omega)$ such that $wk^{p'}\in L^1(\Omega)$ and $|\cdot|$ denotes the Euclidean norm.

\item[($A3$)] $0 < [a(x,\xi) - a(x,\eta)]\cdot (\xi-\eta)$ for a.e. $x\in \Omega $ and all\  $\xi, \eta \in \mathbb R^N,\xi \ne \eta.$

\item[($A4$)] $\bar{C}w(x)|\xi|^{p(x)} \leq a(x,\xi)\cdot \xi$ for a.e. $x\in \Omega $ and all $\xi \in \mathbb R^N,$ where $\bar{C}$ is a positive constant.

\item[($A5$)] $a(x,\xi)\cdot \xi \leq p^+A(x,\xi)$ for a.e. $x\in \Omega $ and all $\xi \in \mathbb R^N.$

\item[($w1$)] $ w \in L_{loc}^{1}(\Omega) $ and $ w^{-s} \in L^{1}(\Omega) $ for some $s\in C(\overline{\Omega })$ such that  $s(x)\in \left(\frac{N}{p(x)},\infty\right)\cap \Big[\frac{1}{p(x)-1},\infty\Big)$ for all $x\in \overline{\Omega }.$
\end{itemize}
The assumption $(w1)$ is to assure basic properties of the weighted variable exponent Sobolev spaces $W^{1,p(x)}(w,\Omega),$ that are defined later. Since $(w1)$ takes place, the weight $w$ may be neither bounded nor away from zero. A problem containing such a $w$ is called degenerate. Notice that $(A0)-(A5)$ may not be fulfilled simultaneously. Also note that under $(A1)$, we have for a.e. $x\in \Omega$ and for all $\xi \in \mathbb R^N$,
\begin{equation}\label{formA}
A(x,\xi)=\int_{0}^{1}a(x,t\xi)\cdot \xi \ dt.
\end{equation}
Thus, it is easy to see that $(A2)$ implies that
\begin{itemize}
\item[($\widetilde{A}2$)] $A(x,\xi) \leq \widetilde{C}_1\left[\widetilde{k}(x)+w(x)|\xi|^{p(x)}\right]$
for a.e. $x\in \Omega$ and all $\xi \in \mathbb R^N,$ where $\widetilde{k}:=wk^{p'}\in L^1(\Omega)\cap P_+(\Omega)$ and $\widetilde{C}_1:=2\widetilde{C};$
\end{itemize}
and $(A4)$ implies that
\begin{itemize}
\item[($\widetilde{A}4$)] $\frac{\bar{C}w(x)}{p(x)}|\xi|^{p(x)}\leq A(x,\xi)$
for a.e. $x\in \Omega$ and all $\xi \in \mathbb R^N.$
\end{itemize}

The operator satisfying $(A0)-(A5)$ is of Leray-Lions type and the typical examples are
$$\operatorname{div}\left[w(x) \sum_{i=1}^n |\nabla u|^{p_i(x)-2}\nabla u\right] ~\mbox{and}~ \operatorname{div}\left[w(x)(1+|\nabla u|^2)^{(p(x)-2)/2}\nabla u\right].$$
Let us name the first operator a multiple degenerate $p(x)$-Laplacian and the second one  a degenerate generalized mean curvature operator. Note that if $k(x)$ in $(A2)$ is a positive constant function then $(A2)$ requires $w\in L^1(\Omega).$

Let us denote
\begin{center}
$p_s(x):=\frac{p(x)s(x)}{1+s(x)},$
\end{center}
where $s$ is given in $(w1)$ and
\[
p^{\ast }_{s}(x):=\begin{cases}
\frac{p(x)s(x)N }{(s(x)+1)N-p(x)s(x) } & \text{if }p_{s}(x)<N, \\
+\infty & \text{if }p_{s}(x) \geq N,
\end{cases}
\]
for all $x\in \overline{\Omega}.$ Furthermore, we assume that
\begin{itemize}
\item[($\widetilde{F}1$)] There exists a constant $C>0$ such that $ | f(x,t)| \leq h(x)+C| t| ^{q(x) -1}$
for a.e. $x\in \Omega $ and all $t\in\mathbb{R}$, where $q\in C_{+}(\overline{\Omega }) $ with
$q(x)<p_{s}^{\ast}(x)$ for all $x\in \overline{\Omega }$, and $h\in P_+(\Omega)$ such that $h^{q'}\in L^{1}(\Omega).$
\end{itemize}

Recently, 
 I.H. Kim and Y.-H. Kim \cite{Kim-Kim} considered the problem \eqref{1.1} with $a(x,\xi)=\phi(x,|\xi|)\xi$ which is of type $|\xi|^{p(x)-2}\xi$ (non-degenerate cases) and a growth condition which is a little different from $(\widetilde{F}1)$. Under suitable conditions on $\phi$ that are a special case of our assumptions, they obtained the existence and multiplicity of solutions using the Mountain Pass Theorem and Fountain Theorem. It is worth noting that the main operator in \cite{Kim-Kim} cannot include multiple $p(x)$-Laplace operators.
 Boureanu and Udrea \cite{Boureanu} considered the problem
\begin{eqnarray}\label{1.4}
\begin{cases}
 -\operatorname{div}a(x,\nabla u) + |u|^{p(x)-2}u =\lambda f(x,u) \quad &\text{in } \Omega ,\\
  u=c\ (constant) \quad &\text{on } \partial \Omega ,
  \end{cases}
\end{eqnarray}
where $a$ satisfies $(A0)-(A5)$ with $w \equiv 1$
and showed the existence and multiplicity of solutions for \eqref{1.4} in the case of both $(p(\cdot)-1)$-sublinear at infinity and $(p(\cdot)-1)$-superlinear at infinity of nonlinearity. Moreover, they used several three solutions type theorems to obtain at least three distinct solutions for \eqref{1.4} for the case of $(p(\cdot)-1)$-sublinear at infinity. For a degenerate case, the authors \cite{Ky1} considered the problem
\begin{eqnarray}\label{1.2}
\begin{cases}
 -\operatorname{div}(w(x) |\nabla u|^{p(x)-2}\nabla u) =f(x,u) \quad &\text{in } \Omega ,\\
  u=0\quad &\text{on } \partial \Omega ,
  \end{cases}
\end{eqnarray}
under the condition $(w1)$ and showed that
\begin{itemize}
\item[1)] the existence of non-trivial solutions using the Mountain Pass Theorem when $f$ satisfies $(\widetilde{F}1)$ with $h$ is constant and $q^- > p^+$, the Ambrosetti-Rabinowitz condition (the (AR) condition), $\lim_{t\to 0}\frac{f(x,t)}{|t|^{p^{+}-1}}=0$ uniformly for $x\in \overline{\Omega }$;
\item[2)] the uniqueness of solutions  using the Browder's Theorem when $f$ satisfies $(\widetilde{F}1)$ and is nonincreasing with respect to the second variable;
\item[3)]  the uniqueness and the nonnegativeness of solutions  using cut-off method when $w  \in L^\infty(\Omega), w^{-s^+} \in L^{1}(\Omega);$ the nonlinearity $f$ is continuous, and nonincreasing with respect to the second variable and $f(x,0) \ge 0,$ for all $x \in \overline{\Omega}.$
\end{itemize}

Motivated by the above results,  we shall consider degenerate $p(x)$-Laplace equations \eqref{1.1} with Leray-Lions type operators which generalize the main operators in \cite{Boureanu,Ky1,Kim-Kim} to obtain their results. We also use a three solutions type theorem to obtain the multiplicity of solutions for \eqref{1.1} when the nonlinearity is $(p(\cdot)-1)$-sublinear at infinity as in \cite{Boureanu}. However, since the boundary condition in the present paper is different from theirs, we need to give a new approach to deal with the nontriviality of solutions.

This paper is organized as follows. In Section 2, we review the weighted variable exponent Lebesgue-Sobolev
spaces and list properties of those spaces. In Section 3, we obtain variational principles for our variational settings. In Section 4, we show the existence and multiplicity of solution for \eqref{1.1} in two cases; $(p(\cdot)-1)$-superlinear at infinity and  $(p(\cdot)-1)$-sublinear at infinity using direct methods and critical point theories in Calculus of Variations. The final section is devoted to showing that the unique solution of \eqref{1.1} is nontrivial and nonnegative.

\section{Abstract framework and preliminary results}
In this section, we only review  the weighted variable exponent Lebesgue-Sobolev
spaces $L^{p(x) }(w,\Omega ) $ and $W^{1,p(x)}(w,\Omega ) $, which were studied in \cite{Ky1,Kim}  and for the variable exponent Lebesgue-Sobolev
spaces $L^{p(x) }(\Omega ) $ and $W^{1,p(x)}(\Omega ),$ we refer to \cite{Diening,Kovacik} and the references therein.

Let $p\in C_{+}(\overline{\Omega }) $ and $w\in P_+(\Omega),$ we define the variable exponent Lebesgue space as
\[
L^{p(x) }(w,\Omega ) =\left\{  u:\Omega
\to\mathbb{R}\text{ is measurable, }\int_{\Omega }w(x)| u(x)| ^{p(x) }dx<\infty \right\} .
\]
Then $L^{p(x) }(w,\Omega ) $ endowed with the norm
\[
| u| _{L^{p(x)}(w,\Omega )}=\inf\left\{ \lambda >0:\int_{\Omega }w(x)\Big| \frac{u(x) }{\lambda }
\Big| ^{p(x) }dx\leq 1\right\} ,
\]
becomes a normed space. When $w(x)\equiv 1,$ we have $L^{p(x) }(w,\Omega ) \equiv L^{p(x) }(\Omega )$ and use the notation $| u| _{L^{p(x)}(\Omega )}$ instead of $| u| _{L^{p(x)}(w,\Omega )}$.

The following propositions will be useful for the next sections. 

\begin{proposition}[\cite{Diening,Kovacik}] \label{prop1}
The space $L^{p(x) }(\Omega )$ is a separable and uniformly convex Banach space, and its conjugate space is $L^{p'(x) }(\Omega )$ where  $1/p(x)+1/p'(x)=1$. For any $u\in L^{p(x)}(\Omega)$ and $v\in L^{p'(x)}(\Omega)$, we have
\begin{equation*}
\left|\int_\Omega uv\,dx\right|\leq\left(\frac{1}{p^-}+
\frac{1}{{(p')}^-}\right)|u|_{L^{p(x)}(\Omega )}|v|_{L^{p'(x)}(\Omega )}\leq\ 2 |u|_{L^{p(x)}(\Omega )}|v|_{L^{p'(x)}(\Omega )}.
\end{equation*}
\end{proposition}
Define the modular $\rho :L^{p(x) }(w,\Omega )$ $ \to \mathbb{R}$ \  by
\[
\rho (u) =\int_{\Omega }w(x)| u(x)| ^{p(x) }dx,\quad
\forall u\in L^{p(x) }(w,\Omega ) .
\]

\begin{proposition}[\cite{Kim}] \label{prop2}
For all $u\in L^{p(x) }(w,\Omega ),$  we have
\begin{itemize}
\item[(i)] $|u|_{L^{p(x)}(w,\Omega )}<1$ $(=1,>1)$
if and only if \  $\rho (u) <1$ $(=1,>1)$, respectively;

\item[(ii)] If \  $|u|_{L^{p(x)}(w,\Omega )}>1$ then  $|u|^{p^{-}}_{L^{p(x)}(w,\Omega )}\leq \rho (u) \leq |u|_{L^{p(x)}(w,\Omega )}^{p^{+}}$;
\item[(iii)] If \ $|u|_{L^{p(x)}(w,\Omega )}<1$ then $|u|_{L^{p(x)}(w,\Omega )}^{p^{+}}\leq \rho
(u) \leq |u|_{L^{p(x)}(w,\Omega )}^{p^{-}}$.
\end{itemize}
Consequently,
$$|u|_{L^{p(x)}(w,\Omega )}^{p^{-}}-1\leq \rho (u) \leq |u|_{L^{p(x)}(w,\Omega )}^{p^{+}}+1,\ \forall u\in L^{p(x) }(w,\Omega ).$$
\end{proposition}

\begin{proposition} [\cite{Ky1}] \label{prop3}
If $u,u_n\in L^{p(x) }(w,\Omega ) $ ($n=1,2,\dots$), then the
following statements are equivalent:
\begin{itemize}
\item[(i)] $\lim_{n\to \infty }|u_n-u|_{L^{p(x)}(w,\Omega )}=0$;

\item[(ii)] $\lim_{n\to \infty }\rho (u_n-u)=0$.

\end{itemize}
\end{proposition}

The weighted variable exponent Sobolev space $W^{1,p(x)}(w,\Omega) $ is defined by
\[
W^{1,p(x)}(w,\Omega ) =\{u\in L^{p(x) }(\Omega) :
|\nabla u|\in L^{p(x) }(w,\Omega ) \},
\]
with the norm
\[
\|u\|_{W^{1,p(x)}(w,\Omega )}=|u|_{L^{p(x)}(\Omega )}+\big||\nabla u|\big|_{L^{p(x)}(w,\Omega )}.
\]
$W^{1,p(x)}_{0}(w, \Omega ) $ is defined as the closure of $C_0^{\infty }(\Omega )$
in $W^{1,p(x)}(w,\Omega )$ with respect to
the norm $\| \cdot\| _{W^{1,p(x)}(w,\Omega)}$.

The separability of $W^{1,p(x) }(w,\Omega ) $ is required for Fountain Theorem.

\begin{proposition} \label{prop5}
Assume that $(w1)$ holds. Then $W^{1,p(x) }(w,\Omega ) $ is a separable reflexive Banach space.
\end{proposition}
\begin{proof}
By Theorem 2.10 in \cite{Kim}, we have that $W^{1,p(x) }(w,\Omega ) $ is a reflexive Banach space. The only thing left to prove is the fact that $W^{1,p(x) }(w,\Omega ) $ is separable. We first show that $L^{p(x)}(w,\Omega)$ is separable. It is well-known that $L^{p(x)}(\Omega)$ is separable so there exists a countable subset $\mathcal{F}$ of $L^{p(x)}(\Omega)$ such that $\mathcal{F}$ is dense in $L^{p(x)}(\Omega)$. Let $\mathcal{F}_w=\{w^{-\frac{1}{p}}f: f\in \mathcal{F}\}$. It is clear that $\mathcal{F}_w$ is a countable subset of $L^{p(x)}(w,\Omega)$. For any $f\in L^{p(x)}(w,\Omega)$, we have  $w^{\frac{1}{p}}f \in L^{p(x)}(\Omega)$. So there exists a sequence $\{f_n\}_{n=1}^{\infty}\subset \mathcal{F} $ such that
$$|f_n-w^{\frac{1}{p}}f|_{L^{p(x)}{(\Omega)}}\to 0 \ \ \text{as}  \ \ n \to \infty.$$
By this and Proposition \ref{prop3}, we have
$$\int_{\Omega}|f_n(x)-w(x)^{\frac{1}{p(x)}}f(x)|^{p(x)}dx=\int_{\Omega}w(x)|w(x)^{-\frac{1}{p(x)}}f_n(x)-f(x)|^{p(x)}dx\to 0 $$
as $n \to \infty.$ Equivalently,
$$|w^{-\frac{1}{p}}f_n-f|_{L^{p(x)}{(w,\Omega)}}\to 0 \ \ \text{as}  \ \ n \to \infty.$$
Note that $\{w^{-\frac{1}{p}}f_n\}_{n=1}^{\infty}\subset \mathcal{F}_w $. This implies the separability of $L^{p(x)}(w,\Omega)$.

We next show the separability of $W^{1,p(x)}(w,\Omega)$. We have that the product space $Y:=L^{p(x)}(\Omega)\times \left(L^{p(x)}(w,\Omega)\right)^N=L^{p(x)}(\Omega)\times L^{p(x)}(w,\Omega)\times \cdots \times L^{p(x)}(w,\Omega)$ is a separable Banach space with 
an equivalent norm
$$|(u_0,u_1,\cdots,u_N)|_{Y}=|u_0|_{L^{p(x)}(\Omega)}+\left|\left(\sum_{i=1}^{N}u_i^2\right)^{1/2}\right|_{L^{p(x)}(w,\Omega)}$$
Consider the operator
$$ T: W^{1,p(x)}(w,\Omega) \to (Y,|\cdot|_Y),\ Tu=(u,u_{x_1},\cdots ,u_{x_N}).$$
Then $T$ is a linear isometric operator.
Obviously, $T\left(W^{1,p(x)}(w,\Omega)\right)$ is closed in $Y$. Indeed, if $\{u_n\}\subset W^{1,p(x)}(w,\Omega)$ and $Tu_n \to v=(v_0,v_1,\cdots,v_N)$ as $n \to \infty$ in $Y$ i.e.,
\begin{equation} \label{convergence}
\begin{cases}
u_n \to v_0 \ \ &\text {in} \ \ L^{p(x)}(\Omega),\\
\frac{\partial u_n}{\partial x_i}\to v_i \ \ &\text {in} \ \ L^{p(x)}(w,\Omega)
\end{cases}
\end{equation}
as $n \to \infty$ for each $i=1,\cdots ,N$. Then we have that $\{u_n\}$ is a Cauchy sequence in $W^{1,p(x)}(w,\Omega)$ so there exists $u \in W^{1,p(x)}(w,\Omega)$ such that $u_n \to u$ as $n \to \infty$ in $W^{1,p(x)}(w,\Omega)$ and hence, $v_0=u$. We now show that $v_i=\frac{\partial u}{\partial x_i}$. In fact, for any $\phi \in C_0^{\infty }(\Omega )$ we have
\begin{equation}\label{weakderivative}
\int_{\Omega}\frac{\partial u_n}{\partial x_i}\phi dx=-\int_{\Omega}u_n\frac{\partial \phi }{\partial x_i} dx.
\end{equation}
It follows from \eqref{convergence} that $\int_{\Omega}u_n\frac{\partial \phi }{\partial x_i} dx \to \int_{\Omega}u \frac{\partial \phi }{\partial x_i} dx $ as $n \to \infty$. Let $\Omega_0=\operatorname {supp}(\phi)$ then $L^{p(x)}(w, \Omega) \hookrightarrow L^1(\Omega_0)$ (see \cite [Proposition 2.8]{Kim}). So \eqref{convergence} implies that $\frac{\partial u_n}{\partial x_i}\to v_i$ in $L^1(\Omega_0)$. Hence, we infer $\int_{\Omega}\frac{\partial u_n}{\partial x_i}\phi dx \to \int_{\Omega}v_i\phi dx $ as $n\to \infty$ since $$\left|\int_{\Omega}\frac{\partial u_n}{\partial x_i}\phi dx-\int_{\Omega}v_i\phi dx\right|\leq \int_{\Omega}\left|\frac{\partial u_n}{\partial x_i}-v_i\right||\phi| dx\leq \|\phi\|_{\infty}\int_{\Omega_0}\left|\frac{\partial u_n}{\partial x_i}-v_i\right| dx.$$
Letting  $n\to \infty,$ we obtain from these facts and \eqref{weakderivative}  that
\begin{equation*}
\int_{\Omega}v_i\phi dx=-\int_{\Omega}u\frac{\partial \phi }{\partial x_i} dx.
\end{equation*}
So $v_i=\frac{\partial u}{\partial x_i}$ and hence, $v=Tu$ \ i.e.,\ $v\in T\left(W^{1,p(x)}(w,\Omega)\right)$. So $T\left(W^{1,p(x)}(w,\Omega)\right)$ is closed in $Y$. It implies that $T\left(W^{1,p(x)}(w,\Omega)\right)$ is separable and hence, so is
$W^{1,p(x)}(w,\Omega)$.
\end{proof}

The next imbedding result will be used in the later sections frequently.

\begin{proposition} [\cite{Kim}] \label{prop6}
Assume that $(w1)$ holds. If $q\in C_{+}(\overline{\Omega })$ and
$q(x) < p^{\ast }_{s}(x) $ for all $x\in \overline{\Omega }$,
then we obtain the continuous  compact imbedding
\begin{center}
$W^{1,p(x)}(w,\Omega) \hookrightarrow \hookrightarrow L^{q(x) }(\Omega ) .$
\end{center}
\end{proposition}


\section{Variational settings}


Throughout the paper, let us denote $X:=W^{1,p(x)}_{0}(w,\Omega)$ and on  $X$ we have an equivalent norm $\| u\| = \big| | \nabla u| \big|_{L^{p(x)}(w,\Omega)}$
(see \cite[Corollary 2.12]{Kim}). Using the same argument as in the proof of Theorem 4.1 of \cite{Le}
we have the next lemma, the so called $(S_+)$-property which will be used to show compactness.
\begin{lemma} \label{S+fora}
Assume that $(w1),(A2)-(A4)$ hold. If \ $u_{n}\rightharpoonup u$ (weakly) as $n\to\infty$ in  $X$ and
$$\limsup_{n\to  \infty } \int_{\Omega }a(x,\nabla u_n)\cdot(\nabla u_{n}-\nabla u) dx\leq 0$$  then  $u_{n}\to  u$ (strongly) as $n\to\infty$ in $X$.
\end{lemma}
Let us define $\Phi, \Psi, J:X \to\mathbb{R}$ by
\begin{equation}\label{form_of_oper}
\Phi(u) =\int_{\Omega }A(x,\nabla u)dx,\ \Psi (u)=\int_{\Omega}F(x,u)dx,  \ \ \text{and} \ \ J=\Phi- \lambda\Psi
\end{equation}
with $ F(x,t)=\int_0^{t}f(x,s)ds.$ Then we have fundamental results for $\Phi$ and $\Psi.$ 

\begin{lemma} \label{diff}
\begin{itemize}
\item[(i)]
Assume that $(w1), (A1), (A2)$ hold. Then $\Phi \in C^{1}(X,\mathbb{R})$ and
$$
\langle \Phi'(u) ,\upsilon \rangle =\int_{\Omega
}a(x,\nabla u)\cdot \nabla
\upsilon\ dx,\ \ \text{for any}\ \ u,\upsilon \in X.$$
If in addition $(A3)-(A4)$ hold then $\Phi': X \to X^\ast$ is a homeomorphism  with a bounded inverse.
\item[(ii)]Assume that $(w1), (\widetilde{F}1)$ hold.
Then $\Psi\in C^{1}(X,\mathbb{R})$ and
$$
\langle \Psi'(u) ,\upsilon \rangle =\int_{\Omega }f(x,u)\upsilon\ dx, \ \ \text{for any} \ \ u,\upsilon \in X.$$
Moreover, $\Psi$ and $\Psi'$   are sequentially weakly continuous, i.e., \ $u_{n}\rightharpoonup u$ (weakly) as $n\to\infty$ in $X$ implies $\Psi(u_n) \to \Psi(u)$  and $\Psi'(u_n) \to \Psi'(u)$ as $n\to\infty$ in $\mathbb{R}$ and $X^\ast,$ respectively.
\end{itemize}

\end{lemma}
\begin{proof}
By modifying the proof of Lemma 3.1 in \cite{Ky1}, we can easily obtain the differentiability of $\Phi$ and $\Psi$ and their derivative formulas.

In the case of (i), if we assume in addition that $(A3)-(A4)$ hold then $L:=\Phi': X \to X^\ast$ is strictly monotone, coercive, and continuous on $X$. Therefore  by invoking the Browder's theorem for monotone operators in the reflexive Banach spaces (see {\cite[Theorem 26.A]{Zeidler}}), we deduce that $L $ has a bounded inverse $L^{-1}: X^\ast \to X$. Let $f_n \to f$ as $n\to\infty$ in $X^\ast$ and set $u_n=L^{-1}(f_n),u=L^{-1}(f),$ i.e., $f_n=L(u_n), f=L(u)$. 
Then the boundedness of $L^{-1}$ and $\{f_n\}$ imply that $\{u_n\}$ is bounded. So up to a subsequence, we have
$u_{n}\rightharpoonup \bar{u}$ (weakly) as $n\to\infty$ in $X.$ By this and the estimate below
$$|\langle f_n-f,u_n-u\rangle|\leq \|f_n-f\|_{X^\ast}\|u_n-u\|,$$
we infer that $\lim_{n\to  \infty }  \langle L(u_n) ,u_n-u \rangle=
\lim_{n\to  \infty }  \langle f_n ,u_n-u \rangle=\lim_{n\to  \infty }\langle f_n-f,u_n-u\rangle=0,$
i.e., $\lim_{n\to  \infty }\int_{\Omega }a(x,\nabla u_n)\cdot(\nabla u_{n}-\nabla u) dx= 0.$
 This implies that $u_{n} \to \bar{u}$ (strongly) as $n\to\infty$ in $X$ in the view of Lemma~\ref{S+fora}. This yields $f_n=L(u_n) \to L(\bar{u})$ and thus, $f=L(\bar{u})$. By the strict monotonicity of $L$ we obtain $u=\bar{u}$ and hence, $L^{-1}(f_n)\to L^{-1}(f),$ i.e., $L^{-1}$ is continuous on $X^\ast$.

Areguments for showing the sequentially weak continuity of $\Psi$ and $\Psi'$ in the case of (ii) are standard (see e.g., \cite[Proof of Proposition 2.9]{Ky2}), we omit it.
\end{proof}

Lemma~\ref{diff} implies that when $(w1), (A1), (A2)$ and  $(\widetilde{F}1)$ hold, $J\in C^{1}(X,\mathbb{R})$. On more assumptions, we have the following result for $J$ and $J'$. The proof is easily obtained from Lemmas~\ref{S+fora} and \ref{diff}.


\begin{lemma} \label{S+}
Assume that $(w1),(A1)-(A4)$ and $(\widetilde{F}1)$ hold. Then for every $\lambda\in \mathbb{R},$ $J$ is
sequentially weakly lower semicontinuous and the derivative $J'$ is an $(S_+)$ type operator, i.e., if \ $u_{n}\rightharpoonup u$ (weakly) in  $X$ and $\limsup_{n\to  \infty }  \langle J'(u_n) ,u_n-u \rangle\leq 0$  then  $u_{n}\to  u$ (strongly) in $X$.
\end{lemma}


\begin{definition}
We say that $ u\in X $ is a weak solution of \eqref {1.1} if
\begin{equation*} \label{weak_sol_eq}
 \int_{\Omega}a(x,\nabla u(x))\cdot \nabla \varphi(x) dx
=\lambda\int_{\Omega}f(x,u(x))\varphi(x) dx
\end{equation*}
for all $\varphi\in X$.
\end{definition}
Obviously, under assumptions of Lemma~\ref{diff}, a critical point of $J$ is a weak solution of \eqref{1.1}. We now remind some variational principles for the multiplicity of solutions. The following is a three critical points type theorem due to Bonanno et al. \cite{Bonanno2}.
\begin{theorem} [{\cite[Theorem 2.1]{Bonanno2}}]\label{three}
Let $X$ be a reflexive real Banach space, and let $\Phi : X \to \mathbb{R}$ be a coercive, continuously G\^ateaux
differentiable and sequentially weakly lower semicontinuous functional whose G\^ateaux derivative admits a continuous inverse on $X^\ast$, $\Psi : X \to \mathbb{R}$ be a continuously G\^ateaux differentiable functional whose G\^ateaux derivative is compact such that $\Phi(0)= \Psi(0) = 0.$  Assume that there exist $r_0>0$ and $u_0\in X$ with $r_0< \Phi(u_0)$ such that
\begin{itemize}
\item[(i)] $\underset{\Phi(u)<r_0}{\operatorname{sup}}\Psi(u)<r_0\Psi(u_0)/\Phi(u_0);$
\item[(ii)] for each $\lambda\in \Lambda = (\Phi(u_0)/\Psi(u_0), r_0/\underset{\Phi(u)<r_0}{\operatorname{sup}}\Psi(u))$, the functional $\Phi-\lambda\Psi$ is coercive.
\end{itemize}
Then, for each $\lambda\in \Lambda$, the functional $\Phi-\lambda\Psi$ has at least three distinct critical points in X.
\end{theorem}
The last concern is the existence of infinitely many critical points when the energy functional $J$ is symmetric.  Recall that if $X$ is a separable reflexive Banach space then it is well-known that there exist $\{e_n\}_{n=1}^{\infty} \subset X$  and $\{f_n\}_{n=1}^{\infty} \subset X^{\ast}$ such that \ \ $X = \overline{\mbox{span}\{e_n\}_{n=1}^{\infty}}, \quad
X^\ast = \overline{\mbox{span}\{f_n\}_{n=1}^{\infty}}$\ \ and
$$\langle f_i, e_j \rangle =
\begin{cases}
1, & \text{ if } i = j, \\
0, & \text{ if } i \ne j,
\end{cases}$$
where $\langle \cdot,\cdot\rangle$ is the duality product between
$X^\ast$ and $X$ (see \cite[Section 17]{Zhao}) . Denote
$$X_n = \mbox{span}\{e_n\},~~~~~~~~~~Y_n = \oplus_{k=1}^n X_k,~~~~~~~~~~Z_n = \overline{\oplus_{k=n}^\infty X_k}.$$
For $X_n, Y_n, Z_n$ taken as the above, we have

\begin{theorem}[{\cite[Fountain Theorem]{Willem}}] \label{Fountain_thm}
Assume that $J \in C^1(X,\mathbb{R})$ is even and for each $n=1,2,\cdots, $ there exist $\rho_n >\gamma_n>0$ such that
\begin{itemize}
\item[$(H1)$] $b_n=\inf_{\{u \in Z_n:~ \|u\|=\gamma_n\}}J(u) \to +\infty$ as $n \to \infty;$
\item[$(H2)$] $a_n=\max_{\{u \in Y_n:~ \|u\|=\rho_n\}}J(u) \leq 0$;
\item[$(H3)$] $J$ satisfies the $(PS)_c$ condition for every $c>0$.
\end{itemize}
Then $J$ has a sequence of critical values tending to $+\infty$.
\end{theorem}


\section{Existence and multiplicity of Solutions}
We divide this section to show the existence and multiplicity of solutions for \eqref{1.1} into two cases; $(p(\cdot)-1)$-superlinear and $(p(\cdot)-1)$-sublinear at infinity using critical point theories in Calculus
of Variations and the three critical points type theorem~\ref{three}.

\subsection{$(p(\cdot)-1)$-superlinear at infinity}

In this subsection, we shall establish the existence and multiplicity of solutions when $f$ satisfies the $(AR)$ condition. We assume that
\begin{itemize}
\item[($F1$)]   There exists a constant $C>0$ such that
$| f(x,t)| \leq C\left(1+| t| ^{q(x) -1}\right) $ for a.e. $x\in \Omega $
and all $t\in\mathbb{R}$, where $q\in C_{+}(\overline{\Omega }) $ with
$q(x)<p_{s}^{\ast}(x)$ for all $x\in \overline{\Omega }$;
\item[($F2$)] There exist $l>0$ and $\theta >p^{+}$  such that
$$0<\theta F(x,t) \leq f(x,t) t$$ for a.e. $x\in \Omega$ and all $|t| \geq l,$ where $F(x,t)=\int_0^{t}f(x,s)ds$;
\item[($F3$)] $\lim_{t\to 0}\frac{f(x,t)}{|t|^{p^{+}-1}}=0$ uniformly for
$x\in \Omega.$
\end{itemize}
Notice that $(F1)$ implies that
\begin{itemize}
\item[$(G1)$] $|F(x,t)|\leq C\left[|t|+\frac{1}{q(x)}|t|^{q(x)}\right]\leq  C_1\left[1+|t|^{q(x)}\right]$ for a.e. $x\in \Omega $ and all $t\in\mathbb{R}$.
\end{itemize}
Under the condition $(F1),$ the (AR) condition $(F2)$ implies that there exists a function $\kappa \in L^{\infty}(\Omega)\cap P_+(\Omega)$ such that
\begin{equation}\label{4.estF}
F(x,t)\geq \kappa(x)|t|^{\theta},\quad \text{for \ a.e.}\ x\in \Omega\ \ \text{and all}\ \ |t|>l \ \ .
\end{equation}
From this, one can deduce that $f$ is $(p(\cdot)-1)$-superlinear at infinity. The first existence result is obtained by the classical  Mountain Pass Theorem.
\begin{theorem} \label{thm4.1}
Assume that $(w1), (A1)-(A5)$ hold. Assume also that $(F1)-(F3)$ hold in which $p^+ < q^-$. Then  \eqref{1.1} has at least one nontrivial weak solution in $X$ for every $\lambda>0$.
\end{theorem}
If we assume in addition the oddivity on $a$ and $f,$ we can obtain the following type multiplicity of solutions.

\begin{theorem} \label{thm4.2}
Assume that $(w1), (A0)-(A5)$ hold. Assume also that $(F1), (F2)$ hold. If $f(x,-t)=-f(x,t)$ for a.e. $x \in \Omega$ and for all $t \in \mathbb R,$ then $J$ has a sequence of critical points $\{\pm u_n\}$ such that $J(\pm u_n) \to +\infty$ as $n \to \infty$ and \eqref{1.1} has infinitely many pairs of solutions for every $\lambda>0$.
\end{theorem}
The proofs of Theorems~\ref{thm4.1} and \ref{thm4.2} will be done thanks to the next two lemmas which verify the $(PS)$ condition and the mountain pass geometries of $J$, respectively. Since the proofs of these two lemmas are similar to that of Theorem 3.3 in \cite{Ky1}, we omit it.

\begin{lemma} \label{lePS}
Assume that $(w1), (A1)-(A5)$ hold. Assume also that $(F1), (F2)$ hold. Then $J$ satisfies the $(PS)$ condition for every $\lambda>0$.
\end{lemma}


\begin{lemma} \label{leGeo}
Assume that $(w1), (\widetilde{A}2)$ and $(\widetilde{A}4)$ hold. Assume also that $(F1)-(F3)$ hold in which $p^+ < q^-$. Then for every $\lambda>0$
\begin{itemize}
\item [(i)] there exist $r, \rho>0$ such that $J(u)\geq \rho$ if $\|u\|=r$;
\item [(ii)] there exists $e\in X$ with $\|e\|>r$ such that $J(e)<0$.
\end{itemize}
\end{lemma}

We now give proofs of Theorems~\ref{thm4.1} and \ref{thm4.2}.

\begin{proof}[\textbf{Proof of Theorem~\ref{thm4.1}}]

The fact $J(0)=0$ and Lemmas \ref{lePS} and \ref{leGeo} show that $J$ satisfies all required conditions of the Mountain Pass Theorem and this completes the proof.

\end{proof}

\begin{proof}[\textbf{Proof of Theorem~\ref{thm4.2}}]
 We showed that $J\in C^1(X,\mathbb R)$ by Lemma~\ref{diff} and it is obvious to see that $J$ is an  even functional due to the oddivity of $f,$   \eqref{formA} and $(A0)$. Next, we shall verify that $J$ satisfies conditions $(H1)-(H3)$ in Fountain Theorem. $(H3)$ is clearly satisfied by Lemma~ \ref{lePS}.

 To show $(H1)$, we let
  $$\beta_n=\sup \{|u|_{L^{q(x)}(\Omega)}: u\in Z_n,\ \|u\|=1\}$$
  for each $n\in \mathbb{N}.$ By Proposition~\ref{prop6}, it is clear that $\Psi(u)=|u|_{L^{q(x)}(\Omega)}$ is sequentially weakly continuous on $X$. Thus, $\beta_n\to 0$ as $n\to \infty$ in view of \cite[Lemma 3.3]{Fan3}. This implies that there exists a positive integer $n_0$ such that $\beta_n\leq 1$ for all $n\geq n_0$. For each $n\in \mathbb{N},$ define $\gamma_n$ by
$$\gamma_n =
\begin{cases}
\frac{1}{\sqrt{\beta_n}}, & \text{ if } n_0\leq n, \\
1, & \text{ if } 1\leq n <n_0.
\end{cases}$$
Obviously, $\gamma_n\geq 1$ for all $n$ and $\gamma_n \to \infty$ as $n\to \infty.$ Then for $u\in Z_n$ with $\|u\|=\gamma_n$, by $(\widetilde{A}4),(G1)$, Proposition~\ref{prop2} and the definition of $\beta_n$, we have
\begin{align*}
J( u) &\geq \frac{\bar{C}}{p^+}\int_{\Omega} w(x)| \nabla u| ^{p(x)}dx-\lambda\int_{\Omega} F( x,u) dx\\
&\geq \frac{\bar{C}}{p^{+}}\|u\|^{p^{-}}-\lambda C_1\int_{\Omega} |u|^{q(x)}dx-\lambda C_1|\Omega|\\
&\geq \frac{\bar{C}}{p^{+}}\|u\|^{p^{-}}-\lambda C_1\left(1+|u|^{q^+}_{L^{q(x)}(\Omega)}\right)-\lambda C_1|\Omega|\\
&\geq \frac{\bar{C}}{p^{+}}\|u\|^{p^{-}}-\lambda C_1\beta_n^{q^+}\|u\|^{q^+}-\lambda (1+|\Omega|)C_1\\
&\geq \frac{\bar{C}}{p^{+}}\|u\|^{p^{-}}-\lambda C_1\beta_n^{q^+}\|u\|^{q^++p^-}-\lambda (1+|\Omega|)C_1.
\end{align*}
Thus, for $n\geq n_0$ we have
\begin{equation*}\label{estJ}
b_n=\inf_{\{u \in Z_n:~ \|u\|=\gamma_n\}}J(u)\geq \beta_n^{-\frac{p^-}{2}}\left(\frac{\bar{C}}{p^{+}}-\lambda C_1\beta_n^{\frac{q^+}{2}}\right)-\lambda(1+|\Omega|)C_1.
\end{equation*}
This completes $(H1).$

 Finally, we shall verify $(H2)$.
By $(G1)$ and \eqref{4.estF}, there exists a constant $C_{2}$ such that
\[
F(x,t)\geq \kappa(x)|t| ^{\theta }-C_{2},\quad
\text{for a.e.}\ \ x\in \Omega \ \ \text {and all}\ \ t\in \mathbb{R}.
\]
Using this and $(\widetilde{A}2)$, for $u\in Y_{n}$ with \ $\| u\|>1$, we have 
\begin{align}
J ( u)
&\leq \widetilde{C}_1 \int_{\Omega}\left(\widetilde{k}(x)+w(x)|\nabla u|^{p(x)}\right)dx-\lambda\int_{\Omega}\kappa(x)|u|^{\theta}dx+\lambda C_2|\Omega|\notag\\
&\leq \widetilde{C}_1 \|u\|^{p^+}-\lambda|u|^\theta_{L^\theta(\kappa,\Omega)}+\widetilde{C}_1|\widetilde{k}|_{L^1(\Omega)}+\lambda C_2|\Omega|.\label{estJ1}
\end{align}
Since on the finite dimensional space $Y_{n}$, norms $\|\cdot\|$ and $|\cdot|_{L^{\theta}(\kappa,\Omega)} $ are equivalent  and  $p^+ <\theta$, \eqref{estJ1} implies that
$J(u)\leq 0$ for all $u\in Y_n$ with $\|u\|$ large enough. This completes $(H2)$ and thus the proof is done.
\end{proof}

\subsection{$(p(\cdot)-1)$-sublinear at infinity}

In this part we consider problem \eqref{1.1} when $f$ satisfies $(p(\cdot)-1)$-sublinear at infinity $($the condition $(F5)).$ 
We assume that
\begin{itemize}
\item[($F4$)]   $f\in L^\infty(\Omega \times [-T,T])$ for each $T\in \mathbb{R}_+$;
\item[($F5$)] $\lim_{|t|\to \infty}\frac{f(x,t)}{|t|^{p^{-}-1}}=0$ uniformly for $x\in \Omega;$
\item[($F6$)] There exist a constant $t_0>0$ and a ball $B$ with $\overline{B}\subset \Omega$ such that $\int_B F(x,t_0)dx>0$.
\end{itemize}
Using direct methods and critical point theories in Calculus of Variations, we obtain the first result of existence of multiple nontrivial solutions for \eqref{1.1}.

\begin{theorem} \label{thm4.3}
Assume that $(w1), (A1)-(A4)$ and $(F4)-(F6)$ hold. Then there exists $\lambda_0>0$ such that  \eqref{1.1} has at least one nontrivial weak solution for every $\lambda>\lambda_0$. If in addition $(F3)$ holds in which $p^+<p^{\ast}_s(x)$ for all $x\in\overline{\Omega},$ then \eqref{1.1} has at least two nontrivial weak solutions for every $\lambda>\lambda_0$.
\end{theorem}

\begin{proof}
It follows from $(F5)$ that for each $\epsilon>0$, there exists $T=T(\epsilon)>0$ such that
$$|f(x,t)|\leq \epsilon |t|^{p^--1}, \ \text{for a.e.}\ 
x\in \Omega \ \text{and all} \ |t|>T.$$
By this and $(F4)$ we deduce that
\begin{equation}\label{est1f}
|f(x,t)|\leq C(\epsilon)+\epsilon |t|^{p^--1}, \ \text{for a.e.}\  x\in \Omega \ \text{ and all}\ t\in \mathbb{R},
\end{equation}
where $C(\epsilon)=\underset{x\in \Omega, |t|\leq T}{\operatorname{ess \sup}}|f(x,t)|.$ So  by Lemma ~\ref{diff}, $J\in C^{1}(X,\mathbb{R})$ and by Lemma~\ref{S+}, $J$ is sequentially weakly lower semicontinuous. We claim that $J$ is coercive. Indeed, we deduce from \eqref{est1f} that
\begin{equation}\label{est1F}
|F(x,t)|\leq C(\epsilon)|t|+\frac{\epsilon}{p^-}|t|^{p^-}, \ \text{for a.e.}\  x\in \Omega \ \text{ and all}\ t\in \mathbb{R}.
\end{equation}
Since $X \hookrightarrow L^{p^-}(\Omega )\hookrightarrow L^1(\Omega ),$ there are positive constants $C_3, C_4$ such that $|u|_{L^{p^-}(\Omega)}\leq C_3\|u\|,|u|_{L^1(\Omega)}\leq C_4\|u\|$ for all $u\in X$. So for any $u\in X$, $(\widetilde{A}4)$ and \eqref{est1F} yield
\begin{eqnarray*}
J(u)&\geq& \frac{\bar{C}}{p^+}(\|u\|^{p^-}-1)-\lambda C(\epsilon)\int_{\Omega}|u|dx-\frac{\lambda\epsilon}{p^-}\int_{\Omega}|u|^{p^-}dx\\
&\geq& \left(\frac{\bar{C}}{p^+}-\frac{\lambda\epsilon C^{p^-}_3}{p^-}\right)\|u\|^{p^-}-\lambda C(\epsilon)C_4\|u\|- \frac{\bar{C}}{p^+}.
\end{eqnarray*}
Thus, choosing $\epsilon>0$ such that $\frac{\bar{C}}{p^+}-\frac{\lambda\epsilon C^{p^-}_3}{p^-}>0$  we obtain from the last estimate that $J$ is coercive since $p^->1.$ Therefore, $J$ has a global minimizer $u_1$ on $X$.

Next, we show that  there exists  $u\in X$ such that $J(u)<0$. For each  $\epsilon>0,$ let $B_\epsilon:=\{x \in \Omega : \operatorname{dist} (x, B) \leq \epsilon\},$ where $B$ is the ball given in $(F6).$ Take $\epsilon>0$ small enough such that $\overline{B}_\epsilon \subset \Omega$. Then there exists $u_\epsilon \in C_c^1(\Omega)$ such that
$$
u_\epsilon(x):=
\begin{cases} t_0,\quad & x \in B,\\
0, \quad & x \in \Omega \setminus B_\epsilon,
\end{cases}
$$
and $0\leq u_\epsilon(x)\leq t_0, \forall x\in \Omega,$ where $t_0$ is given in $(F6)$. Thus, $u_\epsilon \in X$ and for a.e. $x\in \Omega,$
$$|F(x,u_\epsilon(x))|\leq \int_{0}^{u_\epsilon(x)}|f(x,s)|ds\leq  |f|_{L^\infty(\Omega\times[-t_0,t_0])}u_\epsilon(x)\leq t_0|f|_{L^\infty(\Omega\times[-t_0,t_0])}.$$
Thus, we estimate
\begin{align} \label{est1J}
J(u_\epsilon)
&=\int_\Omega A(x,\nabla u_\epsilon)dx -\lambda\int_\Omega F(x,u_\epsilon(x))dx\nonumber\\
&=\int_\Omega A(x,\nabla u_\epsilon)dx -\lambda\int_{B} F(x,t_0)dx-\lambda\int_{B_\epsilon \setminus B}  F(x,u_\epsilon(x))dx \nonumber \\
&\le \int_\Omega A(x,\nabla u_\epsilon)dx  -\lambda\left[\int_{B}  F(x,t_0)dx-t_0 |f|_{L^\infty(\Omega\times[-t_0,t_0])} |B_\epsilon \setminus B|\right],
\end{align}
where $|B_\epsilon \setminus B|$ is the Lebesgue measure of $B_\epsilon \setminus B$. Fixing a sufficiently small constant $\epsilon_0>0$ such that
$$t_0 |f|_{L^\infty(\Omega\times[-t_0,t_0])} |B_{\epsilon_0} \setminus B|\leq\frac{1}{2}\int_{B}  F(x,t_0)dx,$$
we obtain from \eqref{est1J} that
$$J(u_{\epsilon_0})
\leq \int_\Omega A(x,\nabla u_{\epsilon_0})dx  -\frac{\lambda}{2}\int_{B}  F(x,t_0)dx.$$
 Then  $J(u_{\epsilon_0})<0$ for all  $\lambda> \lambda_0,$ where $\lambda_0:=\frac{2\int_\Omega A(x,\nabla u_{\epsilon_0})dx }{\int_{B}  F(x,t_0)dx}.$ Consequently, for any $\lambda> \lambda_0,$ the  global minimizer $u_1$ satisfies $J(u_1)<0=J(0)$. It means that for any $\lambda > \lambda_0$, the problem \eqref{1.1} has a nontrivial solution  $u_1$.

We now assume in addition that $(F3)$ holds in which $p^+<p^{\ast}_s(x)$ for all $x\in\overline{\Omega}$. Note that the coercivity of $J$ and the $(S_+)$-property of $J'$ imply that  $J$ satisfies the $(PS)$ condition.  We claim that  $J$ also satisfies the geometries in the Mountain Pass Theorem. Indeed, since $p^+<p^\ast_s(x)$ for all $x\in \overline{\Omega },$ there is a constant $q$ such that $p^+<q<p^\ast_s(x)$ for all $x\in \overline{\Omega }$. Thus, we have $X \hookrightarrow \hookrightarrow L^{q }(\Omega )\hookrightarrow L^{p^+}(\Omega )$ in the view of Proposition~\ref{prop6} and the boundedness of $\Omega.$ Let $C_{p^+}$ and $C_{q}$ be two positive imbedding constants such that $|u|_{L^{p^+}(\Omega)}\leq C_{p^+}\|u\|,|u|_{L^{q}(\Omega)}\leq C_{q}\|u\| $ for all $u\in X$.  By $(F3)$ and $(F5)$ we have that for $\delta =\frac{1}{2\lambda p^+ C_{p^+}^{p^+}}$, there exist $\gamma_1>1$ and $\gamma_2>0$ such that
$$|f(x,t)|\leq \delta |t|^{p^--1}\leq \delta |t|^{q-1},\ \text{for a.e.}\ x\in \Omega\ \text{and all}\ |t|>\gamma_1,$$
and
$$|f(x,t)|\leq p^+\delta |t|^{p^+-1},\ \text{for a.e.}\ x\in \Omega\ \text{and all}\ |t|<\gamma_2.$$
Combining these with $(F4)$ we obtain that
\begin{equation} \label{est2F}
|F(x,t)|\leq \delta |t|^{p^+}+C_5|t|^q,\ \text{for a.e.}\ x\in \Omega\ \text{and all}\ t\in \mathbb{R} ,
\end{equation}
where $C_5=C_5(\delta)$ is some positive constant. It implies that for $u\in X$ with $\|u\|<1,$ we have
\begin{eqnarray*}
J(u)&\geq& \frac{\bar{C}}{p^+}\int_{\Omega}w(x)|\nabla u|^{p(x)}dx-\lambda \delta \int_{\Omega}|u|^{p^+}dx-\lambda C_5\int_{\Omega}|u|^qdx\\
&\geq& \frac{\bar{C}}{p^+}\|u\|^{p^+}-\lambda \delta C_{p^+}^{p^+}\|u\|^{p^+}- \lambda C_5 C_q^q\|u\|^q\\
&=& \frac{\bar{C}}{2p^+}\|u\|^{p^+}- \lambda C_5 C_q^q\|u\|^q.
\end{eqnarray*}
Taking $0<r<\min \big\{1,\|u_1\|, \big(\frac{\bar{C}}{2p^+\lambda C_5C_q^q}\big)^{\frac{1}{q-p^+}}\big\}$ and letting $\rho= \frac{\bar{C}}{2p^+}r^{p^+}- \lambda C_5 C_q^q r^q,$ we have
$$J(u)\geq \rho, \forall u\in X\ \text{with}\ \|u\|=r.$$
Thus, $J$ has the second critical point $u_2,$ which is a weak solution to \eqref{1.1} and $J(u_2)\geq \rho>0=J(0);$ so $u_2 \neq u_1$ and $u_2\neq 0$.
\end{proof}
\begin{remark}
Using the similar argument as above with the same assumptions on $a$ and $w$ as in Theorem~\ref{thm4.3}, we can show that \eqref{1.1} has a weak solution in $X$ for every $\lambda>0$ if  $(\widetilde{F}1)$ holds in which $q^+ < p^-.$ That solution  is a global minimizer of $J,$ which is nontrivial  if $(F4),(F6)$ hold in addition and $\lambda$ large enough.
\end{remark}
Next we shall establish the existence of three solutions  for \eqref{1.1} by using Theorem~\ref{three}. Defining $\Phi,\Psi$ as in \eqref{form_of_oper}, we have the following result.
\begin{theorem}\label{thm4.4}
Suppose that $(w1),(A1)-(A4)$ and $(F3)-(F6)$ hold in which $p^+<p^{\ast}_s(x)$ for all $x\in\overline{\Omega}$. Then, there exist $r_0>0$ and $u_0\in X$ with $r_0<\Phi(u_0)$ and $\underset{\Phi(u)<r_0}{\operatorname{sup}}\Psi(u)<r_0\Psi(u_0)/\Phi(u_0)$ such that ~\eqref{1.1} has at least three distinct solutions in X for all $\lambda\in  \Lambda = (\Phi(u_0)/\Psi(u_0), r_0/\underset{\Phi(u)<r_0}{\operatorname{sup}}\Psi(u)).$
\end{theorem}
\begin{proof}
Taking $u_{\epsilon_0}$, $q$ and the imbedding constants $C_q,C_{p^+}$ as in the proof of Theorem~\ref{thm4.3}, we have
$$\Psi(u_{\epsilon_0})=\int_{\Omega}F(x,u_{\epsilon_0})dx>\frac{1}{2}\int_{B}F(x,t_0)dx>0,\ \Phi(u_{\epsilon_0})= \int_{\Omega}A(x,\nabla u_{\epsilon_0})dx>0.$$
Here we note that $\int_{\Omega}A(x,\nabla v)dx=0$ if and only if $v=0.$ Let $r_0$ be such that $0<r_0<\min \{\frac{\bar{C}}{p^+},\Phi(u_{\epsilon_0})\}$. If $\Phi(u)<r_0$ then by $(\widetilde{A}4)$, we get that
$$\frac{\bar{C}}{p^+}\int_{\Omega}w(x)|\nabla u|^{p(x)}dx<r_0,\ \ \text{i.e.,} \ \ \int_{\Omega}w(x)|\nabla u|^{p(x)}dx<\frac{p^+r_0}{\bar{C}}<1.$$
This fact and Proposition~\ref{prop2} yield
$$\|u\|^{p^+}<\frac{p^+r_0}{\bar{C}}<1, \ \ \text{i.e.,} \ \ \|u\|<\big(\frac{p^+r_0}{\bar{C}}\big)^{\frac{1}{p^+}}.$$
Then we can estimate
$$\int_{\Omega}|u|^q dx=|u|_{L^q(\Omega)}^q\leq (C_q\|u\|)^q<\frac{C_q^q(p^+)^{\frac{q}{p^+}}r_0^{\frac{q}{p^+}}}{\bar{C}^{\frac{q}{p^+}}},$$
and
$$\int_{\Omega}|u|^{p^+} dx=|u|_{L^{p^+}(\Omega)}^{p^+}\leq (C_{p^+}\|u\|)^{p^+}<\frac{C_{p^+}^{p^+}p^+r_0}{\bar{C}}.$$
By these estimates and \eqref{est2F}, we have that for $u\in X$ with $\Phi(u)<r_0,$
\begin{eqnarray*}
\int_{\Omega}F(x,u)dx &\leq& \delta \int_{\Omega}|u|^{p^+}dx+C_5\int_{\Omega}|u|^q dx\\
&\leq& \frac{\delta C_{p^+}^{p^+}p^+r_0}{\bar{C}}+\frac{C_5C_q^q(p^+)^{\frac{q}{p^+}}r_0^{\frac{q}{p^+}}}{\bar{C}^{\frac{q}{p^+}}},
\end{eqnarray*}
where $\delta>0$ is arbitrarily fixed. Thus we have
$$\underset{\varPhi(u)<r_0}{\operatorname{sup}}\Psi(u)\leq r_0\left[\frac{\delta C_{p^+}^{p^+}p^+}{\bar{C}}+\frac{C_5C_q^q(p^+)^{\frac{q}{p^+}}r_0^{\frac{q}{p^+}-1}}{\bar{C}^{\frac{q}{p^+}}}\right].$$
Choosing $\delta>0$ such that $\frac{\delta C_{p^+}^{p^+}p^+}{\bar{C}}<\frac{1}{2}\frac{\Psi(u_{\epsilon_0})}{\Phi(u_{\epsilon_0})}$ and  $r_0>0$ small enough such that $\frac{C_5C_q^q(p^+)^{\frac{q}{p^+}}r_0^{\frac{q}{p^+}-1}}{\bar{C}^{\frac{q}{p^+}}}<\frac{1}{2}\frac{\Psi(u_{\epsilon_0})}{\Phi(u_{\epsilon_0})}$, we obtain that
$$\underset{\Phi(u)<r_0}{\operatorname{sup}}\Psi(u)<r_0\frac{\Psi(u_{\epsilon_0})}{\Phi(u_{\epsilon_0})}.$$
Note that $\Phi-\lambda \Psi$ is coercive for every $\lambda \geq 0$ (see the proof of Theorem \ref{thm4.3}). Applying Theorem~\ref{three} with $u_0=u_{\epsilon_0}$ in the view of Lemma~\ref{diff}, we have the conclusion.
\end{proof}

\begin{remark}\label{compare}
To show the nontriviality of solutions in \cite{Boureanu}, they used a stronger condition than $(F6);$
$$F(x,t_0)>0 ~\mbox{ for a.e.}~ x\in \Omega.$$ Moreover, we do not impose $p(x)\geq 2, \forall x\in \overline{\Omega}$ and the uniformly monotonicity of the operator $a$ as in \cite{Boureanu} to obtain Theorem~\ref{thm4.4}.
\end{remark}


\section{Uniqueness and non-negativeness}

In this section, we are concerned with the uniqueness and the nonnegativeness of solutions for \eqref{1.1}.

First, when we assume that $f$ is nonincreasing with respect to the second variable and $a$ holds $(A3),$ we have the unique existence by the Browder's theorem for monotone operators in the reflexive Banach spaces (see {\cite[Theorem 26.A]{Zeidler}}).
Since the proof is very similar to the proof of Theorem 3.4 in \cite{Ky1}, we shall omit it.


\begin{theorem} \label{thm7}
Assume that $(w1), (A2)-(A4)$ hold. Assume also that $(\widetilde{F}1)$  holds and $f:\Omega \times\mathbb{R}\to\mathbb{R}$ is nonincreasing with respect to the second variable. Then for every $\lambda>0$, \eqref{1.1} has a unique weak solution.
\end{theorem}

To show the existence of a unique nontrivial nonnegative solution, we need an $L^\infty$-bound of solutions for \eqref{1.1}. 
For such an $L^\infty$-bound, more restrictions on  $w$ are required.
\begin{itemize}
\item [$(w2)$] $ \ w \in L^{\infty}(\Omega) $ and $ w^{-s^{+}}\in L^{1}(\Omega)$ for some $s\in C(\overline{\Omega })$ such that  $s(x)\in \left(\frac{N}{p(x)},\infty\right)\cap \Big[\frac{1}{p(x)-1},\infty\Big)$ for all $x\in \overline{\Omega }.$
\end{itemize}
\begin{example}
It is easy to verify that $w(x)=|x|^a$ $(0\leq a<\min\{p^-,N(p^--1)\})$ satisfies $(w2)$ with $s(x)\equiv\frac{N-\epsilon}{a},$ where $0<\epsilon<\min\left\{\frac{N(p^--a)}{p^-},\frac{N(p^--1)-a}{p^--1}\right\}.$
\end{example}
We have the following $L^\infty$-bound result.
\begin{lemma}\label{thm8}
Assume that $(w2), (A2), (A4)$ and $(F1)$ hold. Then for every $\lambda>0$, there exist positive constants $\alpha,\beta$ such that, if $u$ is a weak solution of \eqref{1.1}, then $u\in L^\infty (\Omega),$ and
\begin{equation*}
|u|_{L^\infty(\Omega)} \leq \alpha\left[1+\left(\int_{\Omega} |u|^{\widetilde{q}(x) }dx\right)^{\beta}\right],
\end{equation*}
where  $\widetilde{q}(x):= \max \{p(x), q(x)\}$ for all $x\in\overline{\Omega}.$
\end{lemma}
\begin{proof}
Noting $p(x)\leq \widetilde{q}(x)<p_{s}^{\ast}(x),\ \forall x\in \overline{\Omega }$ and $| f(x,t)| \leq 2C\left(1+| t| ^{\widetilde{q}(x) -1}\right)$ for a.e. $x\in \Omega$ and all $t\in\mathbb{R}$ and repeating the argument used in \cite[Proof of Theorem 4.2]{Ky1}, we get the desired conclusion.
\end{proof}

Employing the $L^\infty$-bound  and cut-off method, we obtain the existence of a unique nontrivial nonnegative solution which is an extension of Theorem 5.1 in \cite{Ky1} for a single degenerate $p(x)$-Laplacian. Unlike the assumption in \cite{Ky1}, we do not require the continuity of $f$ for the last main result.
\begin{theorem} \label{thm9}
Assume that $(w2), (A1)-(A4)$ hold. Assume also that $f:\Omega \times\mathbb{R}\to\mathbb{R}$ is nonincreasing with respect to the second variable on $[0,\infty)$ such that $f(x,0)\geq 0,$ $f(x,0)\not\equiv 0$ and $f(\cdot,t)\in L^{\infty}(\Omega)$ for each $t\geq 0$. Then for every $\lambda>0,$ \eqref{1.1} has a nontrivial nonnegative weak solution and this solution is the unique solution if $f$  is nonincreasing with respect to the second variable on $\mathbb{R}.$
\end{theorem}

\begin{proof}
Let us denote $u^- = - \min  \{u,0\}$ for $u\in X$ in this proof. By the hypotheses on $f$, $d:=\underset{x\in \Omega}{\mathop{\rm ess\,sup} }f(x,0)$ is a positive constant. Theorem ~\ref{thm7} guarantees  the existence of a unique weak solution $u_d$ of the problem
\begin{equation}\label{V.1}
\begin{cases}
 -\operatorname{div} a(x,\nabla u)=\lambda d \quad &\text{in } \Omega ,\\
  u=0\quad &\text{on } \partial \Omega.
  \end{cases}
\end{equation}
Lemma~ \ref{thm8} also  guarantees $u_d \in L^{\infty}(\Omega)$. We claim that $u_d$ is nonnegative. In fact, since $u_d$ is a weak solution for \eqref{V.1}, we have
\begin{equation}\label{5.2}
 \int_{\Omega}a(x,\nabla u_d)\cdot \nabla u_d^- dx
=\lambda d\int_{\Omega}u_d^- dx.
\end{equation}
Moreover, from $(A4)$ we have
\begin{eqnarray*}
&{}&\int_{\Omega}a(x,\nabla u_d)\cdot \nabla u_d^- dx=\int_{\{x\in\Omega:\ u_d\geq 0\}}a(x,\nabla u_d)\cdot \nabla u_d^- dx\\
&{}&+\int_{\{x\in\Omega:\ u_d< 0\}}a(x,\nabla u_d)\cdot \nabla u_d^- dx=\int_{\{x\in\Omega:\ u_d< 0\}}a(x,-\nabla u_d^-)\cdot \nabla u_d^- dx\\
&{}&=-\int_{\{x\in\Omega:\ u_d< 0\}}a(x,\nabla (-u_d^-))\cdot \nabla (-u_d^-) dx\leq 0.
\end{eqnarray*}
So it follows from this and \eqref{5.2} that $\int_{\Omega}u_d^{-} dx= 0.$ Consequently, $u_d^{-}=0$ a.e. on $\Omega$, i.e., $u_d$ is nonnegative a.e. on $\Omega$.
Define
\begin{eqnarray*}
f^{\ast}(x,t)=
\begin{cases}
 f(x,0)\quad &\text {if}\ \ x\in \Omega\ \ \text {and}\ \ t<0,\\
 f(x,t)\quad &\text {if}\ \ x\in \Omega\ \ \text {and}\ \ 0\leq t \leq u_d(x),\\
 f(x,u_d(x))\quad &\text {if}\ \ x\in \Omega\ \ \text {and}\ \ u_d(x)<t,
  \end{cases}
\end{eqnarray*}
and
$$F^{\ast}(x,t)=\int_{0}^{t}f^{\ast}(x,s)ds,$$
and
\[
\widetilde{J}(u) =\int_{\Omega }A(x,\nabla u)dx-\lambda \int_{\Omega}F^{\ast}(x,u)dx.
\]
From the assumptions on $f$ and the nonnegativeness of $u_d,$ it is easy to see that there exists a positive constant $C_6$ such that
$$|f^{\ast}(x,t)|\leq C_6 \ \ \text{for a.e.} \ x\in \Omega \ \text{and all} \ t\in\mathbb{R}.$$
So by a similar argument to the proof of Theorem~\ref{thm4.3} we have that $\widetilde{J}\in C^{1}(X,\mathbb{R})$ and
$\widetilde{J}$ is coercive, sequentially weakly lower semicontinuous.
Therefore  $\widetilde{J}$ has a global minimizer $u_{\ast}$ and we have
\begin{eqnarray}
\begin{cases}\label{5.3}
 -\operatorname{div} a(x,\nabla u_{\ast})=\lambda f^{\ast}(x,u_{\ast}) \quad &\text{in } \Omega ,\\
  u_{\ast}=0\quad &\text{on } \partial \Omega.
  \end{cases}
\end{eqnarray}

We claim that $0\leq u_{\ast}\leq u_d$ a.e. on $\Omega$. Indeed, it follows from \eqref{5.3} that
$$\int_{\Omega }a(x,\nabla u_{\ast})\cdot \nabla u_{\ast}^{-} dx=-\int_{\Omega }a(x,\nabla(-u_{\ast}^-))\cdot \nabla (-u_{\ast}^{-}) dx=\lambda\int_{\Omega}f^{\ast}(x,u_{\ast})u_{\ast}^{-}dx.$$
Thus, $u_{\ast}^{-}=0$ because of  $(A4)$ and the fact that $f^{\ast}(x,t)=f(x,0)\geq 0$ for $t<0$. So $u_{\ast}\geq 0$ a.e. on $\Omega$.
To show that $u_{\ast}\leq u_d,$
let $\Omega_1=\{x\in \Omega: u_{\ast} >u_d \}$.
Due to $u_{\ast}\geq 0,$ we have $f^{\ast}(x,u_{\ast})\leq f(x,0)\leq d.$ Thus
\begin{equation}\label{5.4}
-\operatorname{div}\left[ a(x,\nabla u_d)-a(x,\nabla u_{\ast})\right]=\lambda[d-f^{\ast}(x,u_{\ast})] \geq 0.
\end{equation}
Denoting $\varphi=u_d-u_{\ast}$ and taking $-\varphi^{-}$ as a test function in \eqref {5.4} we have
$$-\int_{\Omega} \left[ a(x,\nabla u_d)-a(x,\nabla u_{\ast})\right]\cdot \nabla \varphi^{-} dx\leq 0,$$
and hence, since $\nabla \varphi=-\nabla \varphi^{-}$ on $\Omega_1,$ we get
$$\int_{\Omega_1}  \left[ a(x,\nabla u_d)-a(x,\nabla u_{\ast})\right]\cdot (\nabla u_d-\nabla u_{\ast}) dx\leq 0.$$
It follows that $\nabla \varphi=0$ on $\Omega_1$ due to the monotonicity of $a$. It implies that  $\nabla \varphi^{-}=0$ on $\Omega$ and therefore $\varphi^{-}=0$\  on $\Omega$, i.e., $\varphi \geq 0$\  a.e. on $\Omega$. It means that $u_{\ast}\leq u_d$\  a.e. on\  $\Omega$. Consequently, $f^{\ast}(x,u_{\ast})=f(x,u_{\ast})$ and hence, $u_{\ast}$ is a weak solution of \eqref{1.1}.
Since $ -\operatorname{div}a(x,\nabla u_{\ast})=\lambda f(x,u_{\ast})$ and $f(x,0)\not \equiv 0$ , $u_{\ast}$ is nontrivial.

When $f$  is nonincreasing with respect to the second variable on $\mathbb{R},$ the uniqueness comes from the monotonicity of $a$ and $f$ (see the proof of Theorem 1.1  of \cite{Fan4}).
\end{proof}

\subsection*{Acknowledgement}
The second author was supported by the National Research Foundation of Korea Grant funded by the Korea Government (MEST) (NRF-2015R1D1A3A01019789).

\end{document}